\documentclass[12pt,leqno]{amsart}
\usepackage{times}
\usepackage{amsmath,amssymb,amscd,amsthm}
\usepackage{latexsym}
\usepackage{epsfig}

\newtheorem{theorem}{Theorem}
\newtheorem{lemma}{Lemma}
\newtheorem{proposition}{Proposition}

\newtheorem{corollary}{Corollary}

\newtheorem{claim}{Claim}

\newcommand{\p}{\partial}

\newcommand{\beq}{\begin{equation}}
\newcommand{\eeq}{\end{equation}}
\newcommand{\beqna}{\begin{eqnarray*}}
\newcommand{\eeqna}{\end{eqnarray*}}
\newcommand{\beqn}{\begin{equation*}}
\newcommand{\eeqn}{\end{equation*}}
\newcommand{\bp}{\begin{proof}}
\newcommand{\ep}{\end{proof}}
\newcommand{\bprop}{\begin{proposition}}
\newcommand{\eprop}{\end{proposition}}
\newcommand{\bt}{\begin{theorem}}
\newcommand{\et}{\end{theorem}}
\newcommand{\bex}{\begin{Example}}
\newcommand{\eex}{\end{Example}}
\newcommand{\bc}{\begin{corollary}}
\newcommand{\ec}{\end{corollary}}
\newcommand{\bcl}{\begin{claim}}
\newcommand{\ecl}{\end{claim}}
\newcommand{\bl}{\begin{lemma}}
\newcommand{\el}{\end{lemma}}

\begin{document}

\title
[Periodic Korteweg-de Vries Equation]
{On the periodic Korteweg-de Vries Equation: \\a normal form approach}

\author{Seungly Oh}

\address{Seungly Oh, 
405, Snow Hall, 1460 Jayhawk Blvd. , 
Department of Mathematics,
University of Kansas,
Lawrence, KS~66045, USA}
\date{\today}



\begin{abstract}
This paper discusses an improved smoothing phenomena for low-regularity solutions of the Korteweg-de Vries (KdV) equation in the periodic settings by means of normal form transformation.  As a result, the solution map from a ball on $H^{-\frac{1}{2}+}$ to $C_0^t ([0,T], H^{-\frac{1}{2}+})$ can be shown to be Lipschitz in a $H^{0+}_x$ topology, where the Lipschitz constant only depends on the rough norm~$\|u_0\|_{H^{-\frac{1}{2}+}}$ of the initial data.   A similar episode has been observed in a recent paper on 1D quadratic Schr\"odinger equation in low-regularity setting. 
\end{abstract}

\maketitle
\date{today}

\section{Introduction}
Consider the real-valued periodic Korteweg-de Vries (KdV) equation
\begin{equation}\label{kdvt}
\left|\begin{array}{l} u_t + u_{xxx} = \partial_x ( u^2); \qquad (t,x)\in [0,T]\times \mathbf{T}\\
u(0,x)= u_0 \in H^{-s} (\mathbf{T})\end{array}\right. 
\end{equation}
where $\mathbf{T} = \mathbf{R} \mod [0,2\pi]$.  Using function spaces introduced by Bourgain in \cite{Bour}, the local well-posedness of \eqref{kdvt} for $s< 1/2$ was proved by Kenig, Ponce, Vega in \cite{KPV1}.   Ten years later, the global well-posedness of this problem for $s\leq 1/2$ was proved by Colliander, Keel, Staffilani, Takaoka and Tao in \cite{I1}, where they introduced \emph{I-method} for constructing almost-conserved quantities.  For a further survey on this topic, refer to \cite{Titi} and the references therein.\\

The purpose of this paper is not to extend these results, which are sharp in the sense of uniformly continuous solution map, but instead to examine some properties of solutions with rough initial data.  In fact, we will take for granted the well-posedness of \eqref{kdvt}, but we remark that this method can be refined to produce a priori estimates for the corresponding contraction argument to prove the quantitative local well-posedness of periodic KdV (refer \cite{OS}).  Our main result shows that, at least locally in time, the solution can be decomposed into an explicit component containing the initial data and a \emph{smoother} component.  Furthermore, this leads to a new property of the solution map-- i.e. Lipschitz-ness in a considerably smoother space.  This is a direct analogue of our previous result \cite{OS} for 1-D quadratic Schr\"odinger equation.\\

We achieve this by constructing a bilinear pseudo-differential operator, which \emph{gains} regularity in space essentially by taking advantage of \emph{time resonance} described in \cite{shatah1}.  This technique is known as \emph{normal form}, orginating from the well-known technique in ODE.  It was first applied within the context of PDE by Shatah \cite{shatah} to analyse quadratic non-linear Klein-Gordon equations. More recently, this technique came into much attention, particularly by Germain, Masmoudi and Shatah in \cite{GMS}.\\

This technique was applied in the periodic KdV setting in \cite{Titi}, where it is referred to as \emph{differentiation by parts}.  In \cite{niko}, the differentiation-by-parts technique was used to show a global smoothing for the periodic KdV.  Here, Erdogan and Tzirakis proved that the global solution~$u$ of \eqref{kdvt} with initial data $u_0 \in H^{-s}$ for $s<1/2$ satisfies $u-e^{-t\p_x^3}u_0 \in H^{-s + \gamma}$ where $\gamma <\min (-2s+1,1)$.  Close to $L^2$, this gives a gain of a full derivative.  But as we approach $H^{-1/2}$, this gain gradually disappears.\\

In this paper, we approach the problem in a different manner.  To motivate our setting, we recall from \cite{Bour} that a trilinear resonant term causes a trouble in the low regularity analysis of modified KdV.  A similar trouble was observed for KdV in \cite{Titi} and \cite{niko} after performing differentiation by parts.  In view of this, we will filter out an explicit \emph{roughest resonant} term, denoted $R^*[u_0] \in H^{-s}$, instead of filtering out the free solution~$e^{-t\p_x^3} u_0$ as in \cite{niko}.  The exact formulation of this term will be given in Section~\ref{s31}, as well as the Lipschitz property of this map in Lemma~\ref{lipschitz}.  Then we observe that there can be a smoothing of $1/2$~derivatives even for the lowest known Sobolev index $H^{-1/2+}$.  We remark that this is an improvement from the result in \cite{niko} for the range of indices $s\in (1/4, 1/2)$.\\

The following is the main result of this work:

\begin{theorem}\label{thm2}
Let $0\leq s<1/2$ and $0<\delta \ll \frac{1}{2}-s$.  For any real-valued $u_0 \in H^{-s}(\mathbf{T})$ with $\widehat{u_0}(0) =0$, there exists a time interval~$[0,T]$ with $T \sim \|u(0)\|_{H^{-s}(\mathbf{T})}^{-\alpha}$ for some $\alpha>0$ so that the real-valued solution~$u\in C^0_t([0,T]; H^{-s}(\mathbf{T}))$ to the periodic boundary value problem~\eqref{kdvt} can be decomposed in the following manner: $u = R^*[u_0] + h + w$  where 
$$
R^*[u_0] \in L_t^{\infty} H^{-s}_x;\quad h\in L_t^{\infty} H^{-s+1}_x;\quad w\in X_T^{-s+\frac{1}{2} , \frac{1}{2}+\delta}.
$$

Furthermore, we can write the Lipschitz property of the solution map in a smoother space:

$$
\|u-v\|_{C_t([0,T]; H_x^{-s+\frac{1}{2}})} \leq C_{N,T,\delta} \|u_0 - v_0\|_{H_x^{-s +\frac{1}{2}} },
$$
where $\|u_0\|_{H^{-s}} + \|v_0\|_{H^{-s}} <N$, and $C_{N, T,\delta}$ depends only on $N$, $T$ and $\delta$.
\end{theorem}

We remark that Theorem~\ref{thm2} implies $u-R^*[f] \in C_t^0([0,T]; H^{-s+\frac{1}{2}}_x)$ for $s>1/2$, when we consider the embedding $X^{s,\frac{1}{2}+\delta} \subset C_0^t([0,T]; H^{s}_x)$ for any $\delta>0$ and $s\in \mathbf{R}$.\\

The paper is organized as follows. In Section \ref{s2}, we introduce the $X^{s,b}$ spaces and known results concerning these spaces.  Section~\ref{s3} contains the proof of Theorem~\ref{thm2} in the following manner: In Section~\ref{s31}, we construct the normal form and perform a few change of variables to reach the new formulation of the equation \eqref{kdvt} for our purpose.  In \ref{s32}, we obtain the necessary bilinear and trilinear estimates for the contraction argument. In Section \ref{s33}, we conclude the proof of the theorem.\\

{\bf Acknowledgement:} The author thanks Atanas Stefanov and Vladimir Georgiev for helpful suggestions.

\section{Notations and preliminaries}\label{s2}
\subsection{Notations}\label{s21}
We adopt the standard notations in approximate inequalities as follows:
By $A \lesssim B$, we mean that there exists an absolute constant $C>0$ with $A \leq B$.  $A \ll B$ means that the implicit constant is taken to be a \emph{sufficiently} large positive number.  For any number of quantities $\alpha_1, \ldots, \alpha_k$, $A\lesssim_{\alpha_1, \ldots, \alpha_k} B$ means that the implicit constant depends only on $\alpha_1, \ldots, \alpha_k$.
Finally, by $A\sim B$, we mean $A\lesssim B$ and $B\lesssim B$.\\

We indicate by $\eta$ a smooth time cut-off function which is supported on $[-2,2]$ and equals $1$ on $[-1,1]$.  Notations here will be relaxed, since the exact expression of $\eta$ will not influence the outcome.  For any normed space~$\mathcal{Y}$, we denote the norm $\mathcal{Y}_T$ by the expression $\|u\|_{\mathcal{Y}_T} = \|\eta(t/T)u \|_{\mathcal{Y}}$.\\

The spatial, space-time Fourier transforms and spatial inverse Fourier transform are
\begin{align*}
 \widehat{f}(\xi) &= \int_{\mathbf{T}} f(x) e^{-i x\xi}\, dx, \\
 \widetilde{u}(\tau,\xi) &= \int_{ \mathbf{T}\times\mathbf{R}} u(t,x) e^{-i(x\xi+ t\tau)} \,dx\, dt\\
 \mathcal{F}_{\xi}^{-1}[a_{\xi}](x) &= \frac{1}{2\pi} \sum_{\xi \in \mathbf{Z}} a_{\xi} e^{i\xi x}
\end{align*}
where $\xi \in \mathbf{Z}$.  For a reasonable expression~$\sigma$, we denote pseudo-differential operators with symbol $\sigma (\xi)$ via $\sigma (\nabla) f = \sigma(\p_x) f := \mathcal{F}_{\xi}^{-1} [ \sigma(i\xi) \widehat{f}(\xi)]$.  Also, we define $\langle \xi \rangle := 1+|\xi|$.\\ 

\subsection{$X^{s,b}$ spaces and  and local-wellposedness theory.}
 \label{s22}
Bourgain spaces are constructed as the completion of \emph{smooth} functions with with respect to the norm
$$
\|u\|_{X^{s,b}} := \left( \sum_{\xi\in \mathbf{Z}} (1+ |\xi|)^{2s} (1+ |\tau-\xi^3|)^{2b} |\widetilde{u}(\tau,\xi)|^2 \, d\tau \right)^{\frac{1}{2}}.
$$

 The expression above shows $X^{s,0} = L^2_t H^s_x$.  The added weight~$\tau-\xi^3$ for these spaces naturally related to the linear Airy operator~$\p_t+\p_{xxx}$.  For instance, the free Airy solution with $L^2$~initial data lies in this space, given an appropriate time cut-off~$\eta \in \mathcal{S}_t$.
\begin{equation}\label{airyfree}
\| \eta (t) e^{-t\p_x^3}f\|_{X^{0,b}} = \|(1+| \tau-\xi^3|)^{b} \widehat{\eta}(\tau-\xi^3) \widehat{f}(\xi)\|_{L^2_{\tau} l^2_{\xi}} \lesssim_{\eta, b} \|f\|_{L^2_x}
\end{equation} 
due to the fast decay of $\widehat{\eta} \in \mathcal{S}_t$.  For $\varepsilon, \delta>0$, we have the next two embedding properties (refer \cite{tao2} and Bourgain \cite{Bour}),
\begin{align}
\|u\|_{C^0_t H^s_x (\mathbf{R}\times \mathbf{T})} &\lesssim_{\delta} \|u\|_{X^{s,\frac{1}{2}+\delta}},\label{xsb1}\\
\|u\|_{L^6_{t,x}(\mathbf{R}\times \mathbf{T})} &\lesssim_{\varepsilon, \delta} \|u\|_{X^{\varepsilon,\frac{1}{2}+\delta}}.\label{xsb2}
\end{align}

The following two estimates provide a convenient framework in our proof.  The proofs in \cite[Prop. 2.12, Lemma 2.11]{tao2}, which argues for the non-periodic case ($x\in \mathbf{R}^d$),  can be easily adapted for the periodic case (see, for example \cite[Lemma 7.2]{I1}).  
\begin{proposition}\label{xsb}
$$
\|\eta (t) \int_{0}^t e^{-(t-s)\p_x^3} F(s) \, ds\|_{X^{s,\frac{1}{2}+\delta}} \lesssim_{\eta,\delta} \|u_0\|_{H^s} + \|F\|_{X^{s,-\frac{1}{2}+\delta}}).
$$
\end{proposition}

\begin{proposition} \label{timeloc}
Let $\eta \in \mathcal{S}_t (\mathbf{R})$ and $T\in (0,1)$.  Then for $-\frac{1}{2}< b' \leq b <\frac{1}{2}$, $s\in \mathbf{R}$ and $u\in \mathcal{S}_{t,x}$,
$$
\|\eta(\frac{t}{T}) u \|_{X^{s,b'}} \lesssim_{\eta, b,b'} T^{b-b'} \|u\|_{X^{s,b}}.
$$
\end{proposition}

Using Proposition~\ref{timeloc}, a small positive power of $T$ can be produced by yielding a small portion of the Airy-frequency weight.  We will not explicitly state this gain in effort to simplify notations, but we will assume in the sequel that the implicit constants depending on $T$ will be of \emph{positive} power, unless otherwise stated.\\

To work with functions with mean-zero, we define a closed subspace $Y^{s,b}$ of $X^{s,b}$ (with the same norm) as the image of orthogonal projection $\mathbf{P}: X^{s,b} \to Y^{s,b}$ defined by $\displaystyle \mathbf{P} (u) (x) := u(x) - \int_{\mathbf{T}} u\, dx$.\\

We briefly review the main result of \cite{KPV1} for the periodic KdV.  Kenig, Ponce, Vega proved the bilinear estimate

\begin{equation}\label{kpvper}
\| \p_x (u v) \|_{Y^{-s,-\frac{1}{2}+\delta}}\lesssim_{s,\delta} \|u\|_{Y^{-s,\frac{1}{2}+\delta}} \|v\|_{Y^{-s,\frac{1}{2}+\delta}}
\end{equation}
for $0\leq s <1/2$ and $\delta>0$ small.  Since we are concerned here with local-in-time solution, we characterize the solution~$u$  of \eqref{kdvt} over the time interval~$[0,T]$ by the identity,
\begin{equation}\label{equ}
u = \eta(t) e^{-t\p_x^3} u_0 + \eta(\frac{t}{T}) \int_0^t e^{-(t-s)\p_x^3} \p_x (u^2)(s) \, ds.
\end{equation}

Then Proposition~\ref{xsb} and \ref{timeloc} gives that there exists $\alpha>0$ large so that, for $T\sim \|u_0\|_{H^{-s}}^{-\alpha}$, the contraction argument will hold on a \emph{small} ball in $Y^{-s,\frac{1}{2}+\delta}$ centered at $\eta(t) e^{-t\p_x^3} u_0$ .  In particular, we have $\mathbf{P}u = u$ and $\|u\|_{Y^{-s,\frac{1}{2}+\delta}} \sim \|u_0\|_{H^{-s}}$.\\

Furthermore, if we assume $u$ to be real-valued, then we have $\widehat{u}(-\xi) = \overline{\widehat{u}} (\xi)$.  We will assume this relation in our proof.

\section{Proof of Theorem \ref{thm2}}
\label{s3}

\subsection{Setting of the problem}
\label{s31}
We now turn our attention to \eqref{kdvt}.  Let $u$ be the local solution of \eqref{kdvt} as characterized in \eqref{equ}.  Setting $v := \langle \nabla \rangle^s u$, we have 
\begin{equation}\label{periodic}
v_t + v_{xxx} = \mathcal{N} (v,v), \qquad v(0) = f \in L^2(\mathbf{T})
\end{equation}
where we assume $\int_{\mathbf{T}} f \, dx= 0$ and $\mathcal{N} (u,v) := \p_x \langle \nabla \rangle^{-s} [\langle \nabla \rangle^s u \langle \nabla \rangle^s v ]$.  In particular, the bilinear operator~$\mathcal{N}$ contains a spatial derivative, thus $\mathcal{N} \equiv \mathbf{P} \circ \mathcal{N}$ itself has mean-zero.\\

We construct the bilinear pseudo-differential operator $T$ by the formula
$$
T (u,v) := \sum_{\xi_1 \xi_2 (\xi_1+\xi_2) \neq 0} \frac{\langle\xi_1\rangle^s \langle\xi_2\rangle^s}{\langle\xi_1 + \xi_2\rangle^s} \frac{1}{\xi_1 \xi_2} \widehat{u}(\xi_1) \widehat{v} (\xi_2) e^{i(\xi_1 + \xi_2) x}.
$$

Then the Airy operator acts on $T$ in the following manner:

$$
(\p_t + \p_{xxx}) T (u,v) = T ((\p_t + \p_{xxx}) u,v) + T(u,(\p_t+ \p_{xxx}) v) + \mathcal{N} (\mathbf{P}u,\mathbf{P}v).
$$
If we write $h= T (v,v)$ where $v$ solves \eqref{periodic} (recall $v=\mathbf{P}v$) and change variable by $v= h+z$, then $z$ satisfies
\begin{equation}\label{zperiodic}
\left| \begin{array}{l} (\p_t + \p_{xxx}) z = -2 T (\mathcal{N}(v,v), v);\\
	z(0) = f-T(f,f). \end{array}\right.
\end{equation}

For the right side of \eqref{zperiodic}, we note that 
$$
T(\mathcal{N}(v,v),v) = \mathbf{P}\langle \nabla \rangle^{-s} (\mathbf{P}[\langle \nabla \rangle^s v \langle \nabla \rangle^s v] \frac{\langle \nabla \rangle^{s}}{\nabla} v).
$$
  
We adapt the computations in \cite{NTT2} to simplify Fourier coefficients of the above expression as follows.  For $\xi\neq 0$ (recall $\widehat{v}(0)=0$ and $\widehat{v}(-\xi) = \overline{\widehat{v}}(\xi)$),

\begin{align*} 
\mathcal{F}[\mathbf{P}\langle \nabla \rangle^{-s} (\mathbf{P}[\langle \nabla \rangle^s v \langle \nabla \rangle^s v] \frac{\langle \nabla \rangle^{s}}{\nabla} v)](\xi) &= \sum_{\tiny \begin{array}{c}\xi_1+\xi_2\neq 0, \quad\xi_3 \neq 0\\ \xi_1+\xi_2+\xi_3 = \xi\end{array}} \frac{\langle \xi_1\rangle^s \langle\xi_2\rangle^s \langle \xi_3\rangle^{s}}{i\xi_3\langle \xi\rangle^s} \widehat{v}(\xi_1) \widehat{v}(\xi_2) \widehat{v}(\xi_3)\\
	&\hspace{-150pt}= \sum_{\tiny \begin{array}{c}(\xi_1+\xi_2)(\xi_2+\xi_3)(\xi_3+\xi_1)\neq 0\\ \xi_1+\xi_2+\xi_3=\xi, \quad\xi_3 \neq 0 \end{array} } \frac{\langle \xi_1\rangle^s \langle\xi_2\rangle^s \langle \xi_3\rangle^{s}}{i\xi_3\langle \xi\rangle^s} \widehat{v}(\xi_1) \widehat{v}(\xi_2) \widehat{v}(\xi_3) + \frac{\langle \xi\rangle^{2s}}{-i\xi} \widehat{v}(\xi) \widehat{v}(\xi) \widehat{v}(-\xi)\\
	&\hspace{-105pt}+ \sum_{\xi_3\neq 0} \frac{\langle\xi\rangle^s \langle \xi_3\rangle^{2s}\langle \xi\rangle^s}{i\xi_3\langle \xi\rangle^s} \widehat{v}(-\xi_3) \widehat{v}(\xi) \widehat{v}(\xi_3)+ \sum_{\xi_3\neq 0} \frac{\langle \xi\rangle^s \langle\xi_3\rangle^{2s}}{i\xi_3\langle \xi\rangle^s} \widehat{v}(\xi) \widehat{v}(-\xi_3) \widehat{v}(\xi_3)\\
	&\hspace{-150pt}= \sum_{\tiny \begin{array}{c}(\xi_1+\xi_2)(\xi_2+\xi_3)(\xi_3+\xi_1)\neq 0\\ \xi_1+\xi_2+\xi_3=\xi, \quad\xi_3 \neq 0 \end{array} } \frac{\langle \xi_1\rangle^s \langle\xi_2\rangle^s \langle \xi_3\rangle^{s}}{i\xi_3\langle \xi\rangle^s} \widehat{v}(\xi_1) \widehat{v}(\xi_2) \widehat{v}(\xi_3) - \frac{\langle \xi\rangle^{2s}}{i\xi} |\widehat{v}|^2(\xi) \widehat{v}(\xi).
\end{align*}

We say that the first term on the right side of above is \emph{non-resonant} and denoted $\mathcal{NR}(\xi)$, and the second one is \emph{resonant} and denoted $\mathcal{R}(\xi)$.  Then we can rewrite \eqref{zperiodic} as 
$$
(\p_t + \p_{xxx}) z = -2[ \mathcal{F}_{\xi}^{-1} (\mathcal{NR}) + \mathcal{F}_{\xi}^{-1} \mathcal{R} ].
$$

To deal with the \emph{resonant} term, we construct a solution map for the IVP
$$
\left| \begin{array}{l}(\p_t+\p_{xxx})v_* = -2 \sum_{\xi \neq 0} \frac{\langle \xi \rangle^{2s}}{i\xi} |\widehat{v_*}(\xi)|^2 \widehat{v_*} (\xi) e^{i\xi x} \\
	v_*(0) = f \in L^2(\mathbf{T}). \end{array} \right.
$$

We accomplish this by constructing a solution through the map~$R: L^2 \to C_t^0 L^2_x$ defined
\begin{equation}\label{defr}
R [f] (t,x) := \sum_{\xi \neq 0} \widehat{f}(\xi) e^{2i\frac{\langle \xi \rangle^{2s}}{\xi} |\widehat{f}(\xi)|^2 t} e^{i(\xi x + \xi^3 t)}.
\end{equation}

We remark that $R^*[u_0]$ in the statement of Theorem~\ref{thm2} corresponds to $\langle \nabla \rangle^s R[\langle \nabla \rangle^{-s} u_0]$.  For constructions and properties of such solution maps, refer to \cite[Exercise 4.20-21]{tao2}.  We perform another change of variable~$z = R[f] + w$ to obtain the equation for $w$ (recall that now $v=R[f]+ h +w$),
\begin{equation}\label{eqwper}
\left| \begin{array}{l} (\p_t + \p_{xxx}) w = -2\mathcal{F}_{\xi}^{-1} (\mathcal{NR}) +2 \sum_{\xi\neq 0} \frac{\langle \xi \rangle^{2s}}{i\xi} B(\widehat{R[f]}(\xi), \widehat{h}(\xi), \widehat{w}(\xi)) e^{i\xi x}\\
	w(0) = -T(f,f) \in H^1(\mathbf{T})\end{array} \right.
\end{equation}
where $B(x,y,z):= |x+y+z|^2(y+z) + x|y+z|^2+ x^2 (\overline{y+z}) + |x|^2 (y+z)$ for $x,y,z \in \mathbf{C}$.  Heuristically $R[f]$ is the least smooth term among the three, so it is to our benefit that the particular tri-linear form in \eqref{eqwper} excludes the Fourier coefficients~$|\widehat{R[f]}|^2 \widehat{R[f]}$.

\subsection{Estimates for the non-linearities.}
\label{s32}
 In this section, we establish necessary estimates for the contraction argument of \eqref{eqwper} in $Y^{\frac{1}{2},\frac{1}{2}+\delta}$.\\  

First, we examine mapping properties of $T$.  The next lemma gives that $T(f,f) \in H^1(\mathbf{T})$ and also $h\in L^{\infty}_t H^1_x$ due to the embedding~\eqref{xsb1}.

\begin{lemma}\label{tpmap1}
$T: L^2(\mathbf{T}) \times L^2(\mathbf{T}) \to H^{1}(\mathbf{T})$ is a bounded bilinear operator.
\end{lemma}

\begin{proof}
Let $u,v\in C^{\infty}(\mathbf{T})$.  Then
\begin{equation}\label{tpeq}
\| T(u,v)\|_{H^{\frac{1}{2}}} \sim \| \sum_{\xi_1 (\xi-\xi_1)\neq 0} \frac{\langle\xi_1\rangle^s \langle\xi-\xi_1\rangle^s \langle\xi\rangle^{1-s} }{\xi_1 (\xi-\xi_1)} \widehat{u}(\xi_1) \widehat{v}(\xi-\xi_1)\|_{l^2_{\xi}(\mathbf{Z}\setminus \{0\}) }.
\end{equation}

By symmetry, we can assume $|\xi_1| \geq |\xi-\xi_1|$.  Then by H\"older and Sobolev embedding,
\begin{align*}
\eqref{tpeq} &\lesssim  M \|\sum_{\xi_1 \neq \xi} |\widehat{u}|(\xi_1) \frac{|\widehat{v}|(\xi-\xi_1)}{|\xi-\xi_1|^{\frac{1}{2}+\varepsilon}}\|_{l^2_{\xi}(\mathbf{Z}) }\\
	 &\lesssim M \|\mathcal{F}^{-1}[ |\widehat{u}|] |\p_x |^{-\frac{1}{2}-\varepsilon} \mathcal{F}^{-1}[|\widehat{v}|]\|_{L^2_{x}(\mathbf{T})} \lesssim_{\varepsilon} M \| u\|_{L^2(\mathbf{T})} \| v \|_{L^2_{x}(\mathbf{T})}
\end{align*}
where
$$
M := \sup_{\xi \xi_1 (\xi-\xi_1)\neq 0}\frac{\langle\xi_1\rangle^s \langle\xi-\xi_1\rangle^s \langle\xi\rangle^{1-s} }{|\xi_1| |\xi-\xi_1|^{\frac{1}{2}-\varepsilon}}.
$$

It is easy to see that $M$ is a bounded quantity if $s<1/2$, thus the claim follows.
\end{proof}

We derive the necessary estimates for the non-resonant term in the next lemma.

\begin{lemma}\label{nonres}
For $v\in Y^{0,\frac{1}{2}+\delta}$ and $s<1/2$, we have
$$
\|\mathcal{F}_{\xi}^{-1}(\mathcal{NR}) \|_{Y_T^{\frac{1}{2},-\frac{1}{2}+\delta}} \lesssim_{\delta, T} \|v\|^3_{Y^{0,\frac{1}{2}+\delta}}.
$$
\end{lemma}

\begin{proof}
Note that for all the terms above, $\|\cdot\|_{X^{s,b}} = \|\cdot\|_{Y^{s,b}}$.  Thus it will suffice to show the desired estimate with respect to the $X^{s,b}$ norm. \\

For this trilinear estimate, we use the embedding \eqref{xsb2}.
First we localize each variable in terms of its dispersive frequencies, i.e. $\langle \tau_j- \xi_j^3\rangle \sim L_j$ for $j=1,2,3$ and $\langle \tau-\xi^3\rangle \sim L$, where $L, L_j\gtrsim 1$ are dyadic indices.  We only need to insure that the final estimate includes $L_{\max}^{-\varepsilon}$ for some $\varepsilon>0$ so that sum in these indices (and also gain a small positive power of $T$).\\

First consider the case when $L \gg \max (L_1, L_2, L_3)$.  From the identity $\sum_{j=1}^3 (\tau_j - \xi_j^3) = (\tau-\xi^3)+3(\xi_1+ \xi_2)(\xi_2+\xi_3)(\xi_3+\xi_1)$ for every $\sum_{j=1}^3 \xi_j=\xi$ and $\sum_{j=1}^3 \tau_j=\tau$, we can deduce that $L \sim |\xi_1+ \xi_2||\xi_2+\xi_3||\xi_3+\xi_1|\geq 1$.\\

For fixed $L, L_1, L_2, L_3$, apply Plancherel and H\"older, followed by \eqref{xsb2} to obtain

\begin{align*}
& \| \sum_{\tiny \begin{array}{c}(\xi_1+\xi_2)(\xi_2+\xi_3)(\xi_3+\xi_1)\neq 0\\ \xi_1+\xi_2+\xi_3=\xi, \quad \xi_3 \neq 0 \end{array} } \frac{\langle \xi_1\rangle^s \langle\xi_2\rangle^s \langle \xi_3\rangle^{s}\langle \xi\rangle^{\frac{1}{2}-s}}{i\xi_3\langle\tau-\xi^3 \rangle^{\frac{1}{2}-\delta}} [\widetilde{v}(\xi_1)*_{\tau} \widetilde{v}(\xi_2)*_{\tau} \widetilde{v}(\xi_3)](\tau)\|_{L^2_{\tau} l^2_{\xi}(\mathbf{Z}\setminus \{0\})}\\
	&\lesssim_{\delta} M' \| \widetilde{v_{-\delta}} * ( \widetilde{v_{-\delta}} * \widetilde{v_{-\delta}})\|_{L^2_{\tau}l^2_{\xi}} \sim M' \|  (v_{-\delta})^3\|_{L^2_{t,x}} \lesssim M' \| v_{-\delta}\|_{L^6_{t,x}}^3  \lesssim_{\delta} M' \|v\|_{X^{0,\frac{1}{2}+\delta}}^3
\end{align*}
where $\widetilde{v_{-\delta}} (\tau,\xi) := \langle \xi \rangle^{-\delta} |\widetilde{v}|(\tau,\xi)$ and 
\begin{align}
M' &:= \sup_{\tiny \begin{array}{c}(\xi_1+\xi_2)(\xi_2+\xi_3)(\xi_3+\xi_1)\neq 0\\ \xi_1+\xi_2+\xi_3=\xi,\quad \xi_3 \neq 0 \end{array}} \frac{\langle \xi_1\rangle^{s+\delta} \langle\xi_2\rangle^{s+\delta} \langle \xi_3\rangle^{s+\delta}\langle \xi\rangle^{\frac{1}{2}-s}}{|\xi_3| L^{\frac{1}{2}-\delta}} \label{mpdef}\\
	&\lesssim \sup_{\tiny \begin{array}{c}(\xi_1+\xi_2)(\xi_2+\xi_3)(\xi_3+\xi_1)\neq 0\\ \xi_1+\xi_2+\xi_3=\xi \end{array}} \frac{\langle \xi_1\rangle^{s+\delta} \langle\xi_2\rangle^{s+\delta} \langle \xi\rangle^{\frac{1}{2}-s}}{|\xi_3|^{1-s-\delta}(|\xi_1+ \xi_2||\xi_2+\xi_3||\xi_3+\xi_1|)^{\frac{1}{2}-2\delta} L_{\max}^{\delta}}.\notag
\end{align}

We split into two generic cases: 1) $|\xi|\sim |\xi_1|\sim |\xi_2| \sim |\xi_3|$, 2) $|\xi|\sim |\xi_1|\sim |\xi_2|  \gg |\xi_3|$.  Note that the other cases are easier  and naturally follow from the given cases.\\

\begin{itemize}
\item[Case 1.] If $|\xi|\sim |\xi_1|\sim |\xi_2| \sim |\xi_3|$, note $|\xi_1+ \xi_2||\xi_2+\xi_3||\xi_3+\xi_1| \gtrsim \xi$.
$$
M' \lesssim \sup_{\xi} \frac{\langle \xi \rangle^{2s-\frac{1}{2}+3\delta}}{(|\xi_1+ \xi_2||\xi_2+\xi_3||\xi_3+\xi_1|)^{\frac{1}{2}-2\delta}L_{\max}^{\delta}} \lesssim L_{\max}^{-\delta}\sup_{\xi} \langle \xi \rangle^{2s-1+5\delta} \leq L_{\max}^{-\delta}.
$$
	
\item[Case 2.] If $|\xi_1|\sim |\xi_2| \sim |\xi|\gg |\xi_3|$, note $|\xi_1+ \xi_2||\xi_2+\xi_3||\xi_3+\xi_1| \gtrsim \langle \xi\rangle^2$.
$$
M' \lesssim \sup_{\xi} \frac{\langle \xi \rangle^{s+\frac{1}{2}+2\delta}}{(|\xi_1+ \xi_2||\xi_2+\xi_3||\xi_3+\xi_1|)^{\frac{1}{2}-2\delta}L_{\max}^{\delta}} \lesssim L_{\max}^{-\delta}\sup_{\xi} \langle \xi \rangle^{s-\frac{1}{2}+4\delta} \leq L_{\max}^{-\delta}.
$$
\end{itemize}

This concludes our estimate for the case $L \gtrsim \max (L_1, L_2, L_3)$.  On the other hand, if $L_1 \gtrsim \max(L,L_2,L_3)$, we can use the same method as above after a brief justification.  Note that in this case, $L_1 \gtrsim |\xi_1+ \xi_2||\xi_2+\xi_3||\xi_3+\xi_1|\geq 1$.  Thus the same estimates will follow once we can substitute $L_1$ in place of $L$ in \eqref{mpdef}.  The following computations can be used to justify such substitution:
Let $u,v,w \in X^{0,\frac{1}{2}+\delta}$ be localized in frequency space with $L_1\gtrsim L$. 
\begin{align*}
\|uvw\|_{X^{0,-\frac{1}{2}+\delta}} &\sim \|\frac{[\widetilde{u}* (\widetilde{v} *  \widetilde{w})](\tau,\xi)}{\langle \tau - \xi^3\rangle^{\frac{1}{2}-\delta}}\|_{L^2_{\tau} l^2_{\xi}}\\
	&\hspace{-50pt}\sim  \sup_{\|z\|_{L^2_{\tau}l^2_{\xi}}=1} \left| \int_{\tau_1+\tau_2+\tau_3=\tau} \sum_{\xi_1+\xi_2+\xi_3=\xi}\frac{\widetilde{u}(\tau_1 ,\xi_1)\widetilde{v}(\tau_2 ,\xi_2)\widetilde{w}(\tau_3 ,\xi_3)}{\langle \tau-\xi^3 \rangle^{\frac{1}{2}-\delta}} z(\tau,\xi) d\sigma\right| \\
	&\hspace{-50pt}\lesssim \sup_{\|z\|_{L^2_{\tau}l^2_{\xi}}=1} \int_{\tau_1+\tau_2+\tau_3=\tau} \sum_{\xi_1+\xi_2+\xi_3=\xi}\frac{(L_1^{\frac{1}{2}+\delta} |\widetilde{u}|) |\widetilde{v}||\widetilde{w}||z|}{L^{\frac{1}{2}+\delta} L_1^{\frac{1}{2}-\delta}}   d\sigma \\
	&\hspace{-50pt}\lesssim M^* \|u\|_{X^{0,\frac{1}{2}+\delta}} \sup_{\|z\|_{L^2_{\tau}l^2_{\xi}}=1} \| (\frac{|z|}{ \langle \xi \rangle^{\delta}L^{\frac{1}{2}+\delta}} )* ( \widetilde{v_{-\delta}}*\widetilde{w_{-\delta}})\|_{L^2_{\tau_1} l^2_{\xi_1}}\\
	&\hspace{-50pt}\sim M^* \|u\|_{X^{0,\frac{1}{2}+\delta}} \sup_{\|z\|_{L^2_{\tau}l^2_{\xi}}=1} \|\mathcal{F}_{\tau_1,\xi_1}^{-1} \left[\frac{|z|}{ \langle \xi \rangle^{\delta}L^{\frac{1}{2}+\delta}} \right] v_{-\delta} w_{-\delta}\|_{L^2_{t,x}}\\
	&\hspace{-50pt}\lesssim_{\delta} M^* \|u\|_{X^{0,\frac{1}{2}+\delta}}\|v\|_{X^{0,\frac{1}{2}+\delta}}\|w\|_{X^{0,\frac{1}{2}+\delta}}
\end{align*}
where we have used H\"older and \eqref{xsb2} for the penultimate inequality, and
$$	
M^* := \sup_{\xi_1+\xi_2+\xi_3=\xi} \frac{\langle \xi\rangle^{\delta} \langle \xi_2\rangle^{\delta} \langle \xi_3 \rangle^{\delta}}{L_1^{\frac{1}{2}-\delta}}.
$$

Although above computations were done without the pseudo-differential operator for simplicity, it is easy to see that similar arguments can be used to reduce the case~$L_j \sim \max(L,L_1,L_2,L_3)$ for some $j=1,2,3$ to the case $L \gg \max(L_1,L_2,L_3)$.  This concludes the proof.
\end{proof}

The next lemma deals with the \emph{resonant} terms in \eqref{eqwper}.  To reduce the number of cases, we ignore the complex conjugation.  This does not cause any problem in the proof, since we do not take advantage of cancellations from here on. 

\begin{lemma}\label{res}
Let $R[f]$ be as defined in \eqref{defr} and let $h\in L^{\infty}_t H_x^1([0,T] \times \mathbf{T})$, $w\in X^{\frac{1}{2},\frac{1}{2}+\delta}$ be arbitrary.  Then for $s<1/2$, 
\begin{align*}
\|\frac{\langle \xi \rangle^{2s+\frac{1}{2}}}{\xi} \widehat{R[f]}^2 (\widehat{h}+\widehat{w})\|_{L^2_t l^2_{\xi}([0,T]\times\mathbf{Z}\setminus \{0\} )} &\lesssim_{T,\delta} \|f\|^2_{L^2_x} (\|h\|_{L^{\infty}_t H^1_x}+\|w\|_{L^2_t H^{\frac{1}{2}}_x})\\
\|\frac{\langle \xi \rangle^{2s+\frac{1}{2}}}{\xi}\widehat{R[f]} \widehat{h}^2\|_{L^2_t l^2_{\xi}([0,T]\times\mathbf{Z}\setminus \{0\} )} &\lesssim_{T,\delta} T^{\frac{1}{2}} \|f\|_{L^2_x} \|h\|^2_{L^{\infty}_t H^1_x}\\
\|\frac{\langle \xi \rangle^{2s+\frac{1}{2}}}{\xi} \widehat{R[f]} \widehat{w}(\widehat{h}+\widehat{w}) \|_{L^2_t l^2_{\xi}([0,T]\times\mathbf{Z}\setminus \{0\} )} &\lesssim_{T,\delta} T^{\frac{1}{2}} \|f\|_{L^2_x} \|h\|_{L^{\infty}_t H^1_x} (\|h\|_{L^{\infty}_t H^1_x} + \|w\|_{X^{\frac{1}{2},\frac{1}{2}+\delta}})\\
\|\frac{\langle \xi \rangle^{2s+\frac{1}{2}}}{\xi}\widehat{h} \widehat{w}(\widehat{h} +\widehat{w})\|_{L^2_t l^2_{\xi}([0,T]\times\mathbf{Z}\setminus \{0\} )} &\lesssim_{T,\delta} T^{\frac{1}{2}} \|h\|_{L^{\infty}_t H^1_x} \|w\|_{X^{\frac{1}{2},\frac{1}{2}+\delta}} (\|h\|_{L^{\infty}_t H^1_x}+ \|w\|_{X^{\frac{1}{2},\frac{1}{2}+\delta}})\\
\|\frac{\langle \xi \rangle^{2s+\frac{1}{2}}}{\xi}(\widehat{h}^3+\widehat{w}^3)\|_{L^2_t l^2_{\xi}([0,T]\times\mathbf{Z}\setminus \{0\} )} &\lesssim_{T,\delta} \|h\|_{L^{\infty}_t H^1_x}^3+ \|w\|_{X^{\frac{1}{2}+\delta}}^3.
\end{align*}
\end{lemma}

\begin{proof}
Note that $w, h \in L^2_t H^{\frac{1}{2}}_x ([0,T]\times \mathbf{T})$ and 
$$
\|\widehat{R[f]}\|_{L_t^{\infty} l_{\xi}^{\infty} ([0,T]\times\mathbf{Z}\setminus \{ 0\} )} = \| \widehat{f}\|_{L_t^{\infty} l^{\infty}_{\xi} ([0,T]\times\mathbf{Z}\setminus \{ 0\} )} \lesssim \|f\|_{L^2}.
$$

Then H\"older's inequality combined with the above remark immediately proves the first estimate.\\

Also for any smooth function $u$, note $\|\widehat{u}\|_{L^{\infty}_t l^{\infty}_{\xi}} \lesssim \|u\|_{L_t^{\infty} L_x^2}$.  Applying $h \in L_t^{\infty} H^1_x \subset L_t^{\infty} L^2_x$ and $w\in X^{\frac{1}{2},\frac{1}{2}+\delta} \subset L^{\infty}_t L^2_x$, the rest of the estimates above follow by the same method.
\end{proof}

Finally, we establish the Lipschitz continuity of the map $R[f]$ on $L^2(\mathbf{T})$.

\begin{lemma}\label{lipschitz}
Let $R$ be defined as in \eqref{defr} with $s<1/2$ and $\gamma\in \mathbf{R}$.  Then for any $f, g\in L^2(\mathbf{T})$ with $f-g \in H^{\gamma}$,
$$
\|R[f] - R[g]\|_{C^0_t H^{\gamma}_x([0,T]\times \mathbf{T})} \leq C_{N,T} \|f-g\|_{H^{\gamma}(\mathbf{T})}
$$
where $\|f\|_{L^2}+\|g\|_{L^2} < N$.
\end{lemma}

\begin{proof}
First we write $\widehat{f}(\xi) = |\widehat{f}(\xi)| e^{i \alpha_{\xi}}$ and $\widehat{g}(\xi) = |\widehat{g}(\xi)| e^{i \beta_{\xi}}$.  Denote $\theta_{\xi} :=\alpha_{\xi} - \beta_{\xi} $  Then, the Law of cosines, triangle and H\"older's inequality gives
\begin{align*}
\|R[f] - R[g]\|_{C^0_t H^{\gamma}_x([0,T]\times \mathbf{T})} &= \sup_{t\in [0,T]} \|\langle \xi\rangle^{\gamma} (|\widehat{f}| e^{2it\frac{\langle \xi\rangle^{2s}}{\xi}(|\widehat{f}|^2 - |\widehat{g}|^2 )+ i\theta_{\xi} } -  |\widehat{g}| )\|_{l^2_{\xi} (\mathbf{Z}\setminus \{0\})}\\
	&\hspace{-120pt}= \sup_{t\in [0,T]} \|\langle \xi\rangle^{\gamma} \left( |\widehat{f}|^2 + |\widehat{g}|^2 - 2 |\widehat{f}| |\widehat{g}| \cos ( 2t\frac{\langle \xi\rangle^{2s}}{\xi}(|\widehat{f}|^2 - |\widehat{g}|^2 ) + \theta_{\xi} ) \right)^{\frac{1}{2}}\|_{l^2_{\xi} (\mathbf{Z}\setminus \{0\})}\\
	&\hspace{-120pt}\lesssim \|\langle \xi\rangle^{\gamma} (|\widehat{f}|-|\widehat{g}|)\|_{l^2_{\xi}} + 2 \sup_{t\in [0,T]} \|\langle \xi\rangle^{2\gamma}  |\widehat{f}| |\widehat{g}|( 1 - \cos ( 2t\frac{\langle \xi\rangle^{2s}}{\xi}(|\widehat{f}|^2 - |\widehat{g}|^2 )+ \theta_{\xi} )\|^{\frac{1}{2}}_{l^1_{\xi} (\mathbf{Z}\setminus \{0\})}\\
	&\hspace{-120pt}\lesssim \|f-g\|_{H^{\gamma}} + 4\sup_{t\in [0,T]}  \|\langle \xi\rangle^{2\gamma}  |\widehat{f}| |\widehat{g}| \sin^2 \left(t\frac{\langle \xi\rangle^{2s}}{\xi}(|\widehat{f}|^2 - |\widehat{g}|^2 )+ \theta_{\xi} \right)\|^{\frac{1}{2}}_{l^{1}_{\xi} (\mathbf{Z}\setminus \{0\})}.
\end{align*}
Using $\sin^2 (A+B) \lesssim A^2 + \sin^2 B$ and the assumption $s<1/2$, we need to estimate
\begin{align}
& \|\langle \xi\rangle^{2\gamma}  |\widehat{f}| |\widehat{g}| (|\widehat{f}|^2 - |\widehat{g}|^2 )^2\|^{\frac{1}{2}}_{l^{1}_{\xi} (\mathbf{Z}\setminus \{0\})},\label{lip1}\\
&\|\langle \xi\rangle^{2\gamma}  |\widehat{f}| |\widehat{g}| \sin^2 \theta_{\xi} \|^{\frac{1}{2}}_{l^{1}_{\xi} (\mathbf{Z}\setminus \{0\})}.\label{lip2}
\end{align}
The bound for \eqref{lip1} is straight-forward.  By H\"older's and triangle inequalty,
$$
\eqref{lip1} \lesssim \|f\|_{L^2}^{\frac{1}{2}} \|g\|_{L^2}^{\frac{1}{2}} \|\langle \xi \rangle^{\gamma} (\widehat{f} - \widehat{g}) (|\widehat{f}| + |\widehat{g}|)\|_{l^{\infty}_{\xi}} \lesssim \|f\|_{L^2}^{\frac{3}{2}} \|g\|_{L^2}^{\frac{3}{2}} \|f-g\|_{H^{\gamma}}.
$$

For \eqref{lip2}, we apply the Law of sines.  Without loss of generality, we can assume $\theta_{\xi} \in (0,\pi)$.  Noting that the triangle with side-lengths equal to $|\widehat{f}|, |\widehat{g}|, |\widehat{f} - \widehat{g}|$ has the angle~$\theta_{\xi}$ which is opposite to the side with length $|\widehat{f}-\widehat{g}|$, we can deduce that $|\widehat{f}|\sin \theta_{\xi}\leq |\widehat{f} - \widehat{g}|$ and likewise for $|\widehat{g}|$.  Thus,
$$
\eqref{lip2} \leq \|\langle \xi\rangle^{2\gamma} |\widehat{f} - \widehat{g}|^2 \|_{l^1_{\xi}}^{\frac{1}{2}} \sim \|\langle \xi \rangle^{\gamma} |\widehat{f} - \widehat{g}|\|_{l^2_{\xi}} \sim \|f-g\|_{H^{\gamma}}.
$$
\end{proof}

\subsection{Conclusion of the proof of Theorem~\ref{thm2}}
\label{s33}
Now we turn to the proof of Theorem~\ref{thm2}.  Note that the existence of $w \in L^2_{t,x}([0,T]\times \mathbf{T})$ as the solution of \eqref{eqwper} is given by the decomposition $w= v - R[f] - h$.  However, we claim that the solution~$w$ lives in a \emph{smoother} space, that is $Y^{\frac{1}{2},\frac{1}{2}+\delta}$.\\

Let $T \in (0,1/2)$ be small.   We construct the Duhamel map~$\Lambda_T$ on $Y^{\frac{1}{2},\frac{1}{2}-\delta}$ by
\begin{equation}\label{duhamelper}
[\Lambda_T w](t) := -\eta(t) e^{-t\p_x^3} T(f,f) + \eta(\frac{t}{T})\int_0^t e^{-(t-s)\p_x^3}[\mathcal{N}_1 + \mathcal{N}_2]\, ds
\end{equation}
for any $w\in Y^{\frac{1}{2},\frac{1}{2}+\delta}$ where
$$
\mathcal{N}_1 := -2\mathcal{F}^{-1}_{\xi} ( \mathcal{NR}); \qquad 
\mathcal{N}_2 := 2 \sum_{\xi\neq 0}\frac{\langle \xi\rangle^{2s}}{\xi} B( \widehat{R[f]}(\xi), \widehat{h}(\xi), \widehat{w}(\xi)) e^{i\xi x}.
$$

Let $\mathcal{B}$ be a ball in $Y^{\frac{1}{2},\frac{1}{2}+\delta}$ centered at $-\eta(t) e^{-t\p_x^3} T^p(f,f)$ with \emph{small} radius.  Then our aim is to show that, for $T$ small, $\Lambda_T$ is a contraction map on $\mathcal{B}$. \\

From Proposition~\ref{xsb}, we have
$$
\|\Lambda_T w\|_{Y^{\frac{1}{2},\frac{1}{2}+\delta}} \lesssim_{\eta} \|T(f,f)\|_{H^{\frac{1}{2}}}+ \|\mathcal{N}_1\|_{Y_T^{\frac{1}{2},-\frac{1}{2}+\delta}}+ \|\mathcal{N}_2\|_{Y_T^{\frac{1}{2},-\frac{1}{2}+\delta}}.
$$

For, the first term, we apply Lemma~\ref{tpmap1},
\begin{equation}\label{per1}
\|T(f,f)\|_{H_x^{\frac{1}{2}}}\lesssim \|T(f,f)\|_{H^1_x} \lesssim \|f\|_{L^2}^2.
\end{equation}

For the non-linear term~$\mathcal{N}_1$, we use Lemma~\ref{nonres} and $\|v\|_{Y^{0,\frac{1}{2}+\delta}} \sim \|f\|_{L^2}$,
\begin{equation}\label{per2}
\|\mathcal{N}_1\|_{Y_T^{\frac{1}{2},-\frac{1}{2}+\delta}} \lesssim_{\delta,T} \|v\|_{Y^{0,\frac{1}{2}+\delta}}^3 \sim \|f\|_{L^2}^3.
\end{equation}

For the last term~$\mathcal{N}_2$, we apply Lemma~\ref{res} and $\|h\|_{L^{\infty}_t H^1_x} \lesssim \|v\|_{X^{\frac{1}{2}+\delta}}^2 \sim \|f\|_{L^2}^2$,
\begin{equation}\label{per3}
\|\mathcal{N}_2\|_{Y_T^{\frac{1}{2},-\frac{1}{2}+\delta}} \lesssim_{\delta,T} \|\frac{\langle \xi\rangle^{2s+\frac{1}{2}}}{\xi} B(\widehat{R[f]}, \widehat{h}, \widehat{w})\|_{L^2_t l^2_{\xi}([0,T]\times \mathbf{Z}\setminus \{0\})} \lesssim_{\delta,T, \|f\|_{L^2}} \|w\|^3+ \|w\|^2 + \|w\| +1
\end{equation}
where the implicit constant in the last inequality involves a positive power of $\|f\|_{L^2}$.  Thus making $T$ suitably small with respect to $\|f\|_{L^2}$, we note that $\Lambda_T$ is a contraction on $\mathcal{B}$.\\

To show Lipschitz property, let $R[f^k]$, $h^k$, $w^k$ for $k=1,2$ be the corresponding solutions with $f$ replaced by $f^k$. Recall $v^k = R[f^k] + h^k + w^k$.  Then we need to show
$$
\|v^1- v^2 \|_{C_t([0,T]; H_x^{\frac{1}{2}})} \lesssim_{N, \delta} \|f^1 - f^2\|_{H^{\frac{1}{2}}_x}
$$
where $\|f^1\|_{L^2} + \|f^2\|_{L^2} < N$.  For the first term, we apply Lemma~\ref{lipschitz},
$$
\|R[f^1] - R[f^2]\|_{C_t([0,T];H^{\frac{1}{2}}_x)} \lesssim_{N, \delta} \|f^1-f^2\|_{H^{\frac{1}{2}}_x}.
$$

For the second term, we use Lemma~\ref{tpmap1} and Lipschitz map of \eqref{equ} to obtain
$$
\|h^1-h^2\|_{C([0,T]; H^{\frac{1}{2}}_x)} \lesssim \|T(v^1+ v^2, v^1-v^2)\|_{L^{\infty}_t H^1_x} \lesssim \|v^1+v^2\|_{L^2}\|v^1-v^2\|_{L^2} \lesssim_N \|f^1-f^2\|_{L^2}.
$$
 
 From the formulation~\eqref{duhamelper}, we have 
 \begin{align*}
 \|w^1-w^2\|_{C_t([0,T]; H^{\frac{1}{2}}_x)} &\lesssim \|w^1-w^2\|_{Y^{\frac{1}{2},\frac{1}{2}+\delta}}\\
 	&\lesssim \|T(f^1,f^1) - T(f^2, f^2)\|_{H^{\frac{1}{2}}_x} + \sum_{j=1}^2 \|\mathcal{N}^1_j- \mathcal{N}^2_j\|_{Y^{\frac{1}{2},-\frac{1}{2}+\delta}}
 \end{align*}
 where for $\mathcal{N}_j^k$  is defined with respect to the initial data $f^k$ for $k=1,2$.  Then the desired estimate follows from estimates \eqref{per1} through \eqref{per3}.

\end{document}